\documentclass[12pt]{article}
 \usepackage{amsmath}
\usepackage{amsfonts}
\usepackage{adjustbox}
\usepackage{tikz}
\usepackage[title]{appendix}
\usepackage{subfig}
\usepackage{array}
\usepackage{blkarray}
\usepackage{tabu}
\usepackage{nth}
\usepackage[colorlinks=true, allcolors=black]{hyperref}

\usepackage[
  separate-uncertainty = true,
  multi-part-units = repeat
]{siunitx}

\usetikzlibrary{arrows}
\usepackage{epsfig,}
\usepackage{epsfig}
\tikzset{
    vertex/.style = {
        circle,
        draw,
        outer sep = 3pt,
        inner sep = 3pt,
    },edge/.style = {->,> = latex'}
}
\usepackage{amssymb}
\usepackage{enumitem}
\usepackage{amsthm}
\usepackage{dsfont}
\def\m{\mathop{\rm max}}
\def\l{\mathop{\rm min}}

\def\rank{\mathop{\rm rank}}
\newcommand{\rr}{\mathbb{R}}
\newcommand{\1}{\mathbf{1}}

\newcommand{\E}{\mathcal{E}}

\def\det{{\rm det}}

\def\a{\alpha}

\newtheorem{theorem}{Theorem}

\newtheorem{example}{Example}

\newtheorem{lemma}{Lemma}

\begin{document}
\begin{center}
\begin{large}
Steiner distance matrix of caterpillar graphs
\end{large}
\end{center}
\begin{center}
Ali Azimi, R. B. Bapat and Shivani Goel\\
\today
\end{center}
\vskip .2cm
\begin{abstract}
For a connected graph $G:=(V,E)$, the Steiner distance
$d_G(X)$ among a set of vertices $X$ is the minimum size among
all the connected subgraphs of $G$ whose vertex set contains $X$. The $k-$Steiner distance matrix $D_k(G)$ of $G$ is a matrix whose rows and columns are indexed by $k-$subsets of $V$. For $k$-subsets $X_1$ and $X_2$, the $(X_1,X_2)-$entry of $D_k(G)$ is $d_G(X_1 \cup X_2)$. In this paper, we show that the rank of $2-$Steiner distance matrix of a caterpillar graph on $N$ vertices and with $p$ pendant veritices is $2N-p-1$. 
\end{abstract}


{\bf AMS CLASSIFICATION.} 05C50\\

{\bf MSC PRIMARY.} 05C05, 05C12, 05C50

\section{Introduction}
Let $G:=(V,E)$ be a connected graph with vertex set $V$ and edge set $E$. Suppose $u,v \in V$. The distance $d(u,v)$ between $u$ and $v$ is the length of the shortest path connecting $u$ and $v$. The distance $d(u,v)$ can be interpreted as the minimum size of a connected subgraph of $G$ containing both $u$ and $v$. This insight towards the distance between two vertices gives the motivation to extend the concept of distance between more than two vertices, known as Steiner distance. 

Let $X \subseteq V$, containing at least two vertices. The Steiner distance
$d_G(X)$ among the vertices of $X$ is the minimum size among
all the connected subgraphs of $G$ whose vertex set contains $X$. We refer to \cite{mao2017steiner} for more details on Steiner distance in graphs.
Let $k \geq 2$ be an integer. We define the $k-$Steiner distance matrix of $G$, denoted by $D_k(G)$ as follows: The rows and columns of $D_k(G)$ are indexed by $k-$subsets of $V$. Let $X_1$ and $X_2$ be two $k$-subsets of $V$. The $(X_1,X_2)-$entry of $D_k(G)$ is $d_G(X_1 \cup X_2)$. We illustrate the $k-$Steiner distance matrix by the following example.

\begin{example} \label{extree}\rm
Consider the following tree $T$ on $5$ vertices.
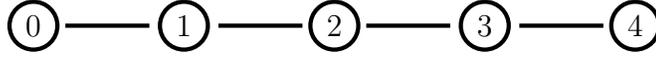
\begin{figure}[!h]
\centering
\begin{tikzpicture}[shorten >=1pt, auto, node distance=3cm, ultra thick,
   node_style/.style={circle,draw=black,fill=white !20!,font=\sffamily\Large\bfseries},
   edge_style/.style={draw=black, ultra thick}]
\node[vertex] (0) at  (0,0) {$0$};
\node[vertex] (1) at  (2,0) {$1$}; 
\node[vertex] (2) at  (4,0) {$2$};  
\node[vertex] (3) at  (6,0) {$3$}; 
\node[vertex] (4) at  (8,0) {$4$};  
\draw  (0) to (1);
\draw  (2) to (1);
\draw  (2) to (3);
\draw  (4) to (3);
\begin{scope}[dashed]
\end{scope}
\end{tikzpicture}
\caption{A tree $T$ on $5$ vertices} \label{fig_tree}
\end{figure}
Suppose the rows and columns of $D_2(T)$ are indexed in the following order:
\[\{0, 1\},
 \{0, 2\},
 \{0, 3\},
 \{0, 4\},
 \{1, 4\},
 \{2, 4\},
 \{3, 4\},
 \{1, 2\},
 \{1, 3\},
 \{2, 3\}\]
Then, the $2-$Steiner distance matrix of $T$ is 
\begin{equation*}\label{steinertree}
    D_2(T) = \left[\begin{array}{cccccccccccc}
1 & 2 & 3 & 4 & 4 & 4 & 4 & 2 & 3 & 3 \\
2 & 2 & 3 & 4 & 4 & 4 & 4 & 2 & 3 & 3 \\
3 & 3 & 3 & 4 & 4 & 4 & 4 & 3 & 3 & 3 \\
4 & 4 & 4 & 4 & 4 & 4 & 4 & 4 & 4 & 4 \\
4 & 4 & 4 & 4 & 3 & 3 & 3 & 3 & 3 & 3 \\
4 & 4 & 4 & 4 & 3 & 2 & 2 & 3 & 3 & 2 \\
4 & 4 & 4 & 4 & 3 & 2 & 1 & 3 & 3 & 2 \\
2 & 2 & 3 & 4 & 3 & 3 & 3 & 1 & 2 & 2 \\
3 & 3 & 3 & 4 & 3 & 3 & 3 & 2 & 2 & 2 \\
3 & 3 & 3 & 4 & 3 & 2 & 2 & 2 & 2 & 1\end{array}\right].
\end{equation*}
\end{example}

Throughout the paper, all vectors are considered as column vectors. For a matrix $A$ and a set of indices $X$, we will use $A[X,X]$ to denote the submatrix of $A$ determined by the rows and columns indexed by $X$. The cardinality of a set $X$ is represented by $|X|$. We will use $J$ to denote the matrix of all ones of appropriate order. For a matrix $A$, $A[i]$ will denote the $i^{\rm th}$ row of $A$. The notations $I$ and $\1$ will represent the identity matrix and all ones vector of the appropriate order, respectively. The degree of a vertex $\alpha$ in a graph is denoted by ${\rm deg} (\alpha)$. 
\subsection{Objective of the paper}
Let $G$ be a connected graph with vertex set $V:=\{0,1,\dotsc,n\}$. The Laplacian matrix of the graph $G$ is the matrix $L := (l_{ij})$, where
\[l_{ij} := \begin{cases}
~~\delta_i &~{\rm if}~i = j\\
-1 &~{\rm if}~i~{\rm and}~j~{\rm are~adjacent}\\
~~0&~{\rm otherwise}.
\end{cases}\]
Here, $\delta_i$ is the degree of the vertex $i$. Let $T$ be a tree with vertex set $\{1,2,\dotsc,n\}$ and distance matrix $D$. Suppose $L$ is the Laplacian of $T$. In \cite{Graham}, Graham and Lov\'asz, showed that 
\[D^{-1} = -\frac{1}{2}L+\frac{1}{2(n-1)}\tau \tau',\]
where $\tau:=(2-\delta_1,\dotsc,2-\delta_n)'$  and $\delta_i$ is equal to the degree of the vertex $i$. In the spirit of Graham and Lov\'asz formula, there are inverse formula for distance matrices of several other connected graphs, see \cite{bapat_kirk}, \cite{sivasu}, \cite{BALAJI2021274}, \cite{balaji2020distance} and \cite{GOEL202186}. We now list the objectives of the paper.
\begin{enumerate}
    \item Let $P_n$ be a path graph on $n+1$ vertices $\{0,1,\dotsc,n\}$ and let
\begin{equation*}\label{setX}
    X := \{\{0,\alpha\}|~\alpha \neq 0\} \cup \{\{\alpha,n\}|~{\rm \deg}(\alpha)\neq 1\}.
\end{equation*}
It is easy to note that $|X| = 2n-1$. 
Motivated by Graham and Lov\'asz formula, we first show that
\[D_2(P_n)[X,X]^{-1} =-L+\frac{1}{n}uu',\]
where $u := (0,\dotsc,0,\underbrace{1}_{n},0,\dotsc,0)' \in \rr^{2n-1}$ and $L$ is the Laplacian of a path graph on $2n-1$ vertices.  

\item A caterpillar graph is a tree in which all the vertices are within a distance $1$ from a central path (see figure \ref{fig_cater}). 
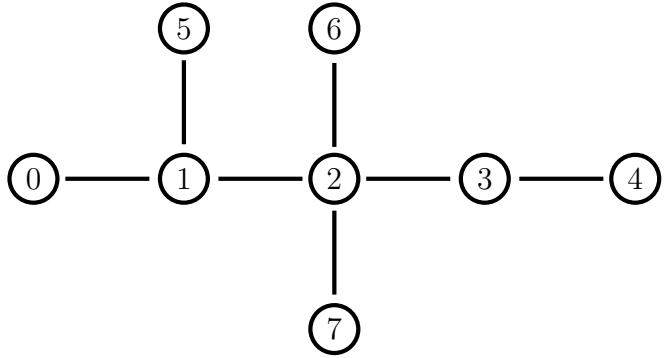
\begin{figure}[!h]
\centering
\begin{tikzpicture}[shorten >=1pt, auto, node distance=3cm, ultra thick,
   node_style/.style={circle,draw=black,fill=white !20!,font=\sffamily\Large\bfseries},
   edge_style/.style={draw=black, ultra thick}]
\node[vertex] (0) at  (0,0) {$0$};
\node[vertex] (1) at  (2,0) {$1$}; 
\node[vertex] (2) at  (4,0) {$2$};  
\node[vertex] (6) at  (4,2) {$6$};  
\node[vertex] (7) at  (4,-2) {$7$};  
\node[vertex] (3) at  (6,0) {$3$}; 
\node[vertex] (4) at  (8,0) {$4$};  
\node[vertex] (5) at  (2,2) {$5$};  
\draw  (0) to (1);
\draw  (2) to (1);
\draw  (5) to (1);
\draw  (2) to (6);
\draw  (2) to (7);
\draw  (2) to (3);
\draw  (4) to (3);
\begin{scope}[dashed]
\end{scope}
\end{tikzpicture}
\caption{A caterpillar graph} \label{fig_cater}
\end{figure}
Suppose $\mathcal{C}$ is a caterpillar graph with its central path of maximum length as $P_n$. We next show that, if 
 \begin{equation*}
    X := \{\{0,\alpha\}|~\alpha \neq 0\} \cup \{\{\alpha,n\}|~deg(\alpha)\neq 1\},
\end{equation*}
then the matrix $D_2(\mathcal{C})[X,X]$ is invertible. In our main result of this paper, we find the rank of $D_2(\mathcal{C})$.

\end{enumerate}

\section{Path graphs} 
Let $P_n$ denotes a path graph on $n+1$ vertices (see Figure \ref{fig_path}). 
\begin{figure}[!h]
\centering
\begin{tikzpicture}[shorten >=1pt, auto, node distance=3cm, ultra thick,
   node_style/.style={circle,draw=black,fill=white !20!,font=\sffamily\Large\bfseries},
   edge_style/.style={draw=black, ultra thick}]
\node[vertex] (0) at  (0,0) {$0$};
\node[vertex] (1) at  (2,0) {$1$}; 
\node[vertex] (2) at  (4,0) {$2$};  
\node[vertex] (n) at  (8,0) {$n$};  
\draw  (0) to (1);
\draw  (2) to (1);
\begin{scope}[dashed]
\draw (2) to (n);
\end{scope}
\end{tikzpicture}
\caption{$P_n$} \label{fig_path}
\end{figure}
Let $D_2(P_n)$ be the Steiner distance matrix of $P_n$ and let
$X := \{\{0,\alpha\}|~\alpha \neq 0\} \cup \{\{\alpha,n\}|~deg(\alpha)\neq 1\}$. In this section, we will deduce an inverse formula for $D_2(P_n)[X,X]$. The result proved in this section will be helpful in proving the main result of the paper. Suppose the sets in $X$ are indexed in the following order:
\[\{0,1\}, \cdots, \{0,n\},\{1,n\},\cdots,\{n-1,n\}.\]
We begin with a few observations on the matrix $D_2(P_n)[X,X]$.
\begin{enumerate}
    \item Let $X_1 = \{0,\alpha_1\}$, $X_2 = \{0,\alpha_2\}$, $X_3 = \{\alpha_3,n\}$ and $X_4 = \{\alpha_4,n\}$, where $0<\alpha_1,\alpha_2 \leq n$ and $0<\alpha_3,\alpha_4< n$ are pairwise distinct vertices of $P_n$. It is easy to see that 
    \[d_{P_n}(X_1 \cup X_2) = \m\{\alpha_1,\alpha_2\},\]
    \[d_{P_n}(X_3 \cup X_4) = n-\l\{\alpha_3,\alpha_4\},\]
and    
    \[d_{P_n}(X_1 \cup X_i) = n, ~{\rm for}~ i = 3,4. \]
    Using the above observation, the submatrix $D_2(P_n)[X,X]$ of $D_2(P_n)$ can be written in the following block form
\begin{equation*}\label{stmatrixpath}
    D_2(P_n)[X,X] = \left[\begin{array}{cc}
         S_1&  nJ_{n,n-1}\\
        nJ_{n-1,n}& S_2
    \end{array}\right],
\end{equation*}
where 
\[S_1 := \left[\begin{array}{ccccccccccccccccccccccc}
        1&2&3&4& \cdots& n\\
        2& 2& 3&4& \cdots& n\\
        3& 3& 3&4& \cdots& n\\
        \cdots& \\
        n& n&n& n& \cdots& n
    \end{array}\right]~\mbox{and}~ 
S_2 := \left[\begin{array}{ccccccccccccccccccccccc}
        n-1&n-1&n-1&\cdots& n-1\\
        n-1&n-2&n-2&\cdots& n-2\\
        n-1&n-2&n-3&\cdots& n-3\\
        \cdots&\\
        n-1&n-2&n-3&\cdots& 1
    \end{array}\right].\]
\item Using the above block form, we deduce expressions for rows of $D_2(P_n)[X,X]$ as follows. If $1\leq i \leq n$, then 
\begin{equation}\label{Di}
    D_2(P_n)[X,X][i]= (\underbrace{i,\dotsc,i}_{i},i+1,i+2,\dotsc,n,\dotsc,n),
\end{equation}
and for $n+1\leq k \leq 2n-1$
\begin{equation}\label{Di1}
    D_2(P_n)[X,X][k]= (\underbrace{n,\dotsc,n}_{n},n-1,n-2,\dotsc,\underbrace{2n-k,\dotsc,2n-k}_{2n-k}).
\end{equation}
\end{enumerate}

In the next two results, we deduce an inverse formula for $D_2(P_n)[X,X]$.
\begin{lemma}
Let $D:= D_2(P_n)[X,X]$. If $L$ is the Laplacian of a path graph on $2n-1$ vertices, then
\[LD+I = u\1',\]
where $u := (0,\dotsc,0,\underbrace{1}_{n},0,\dotsc,0)\in \rr^{2n-1}.$
\end{lemma}
\begin{proof}
The Laplacian matrix of a path graph is a tridiagonal matrix with its main diagonal determined by the vector $(1,2,\dotsc,2,1)$. The first diagonal above and below the main diagonal have all the entries equal to $-1$. Using this, we first note that
\[(LD)[1] = D[1]-D[2],\]
\[(LD)[i] = 2D[i]-D[i-1]-D[i+1],~\mbox{for}~2\leq i \leq 2n-2\]
and 
\[(LD)[2n-1] = D[2n-1]-D[2n-2].\]
We now compute exact expressions for the rows of $LD$ by using (\ref{Di}) and (\ref{Di1}), repeatedly. This is done by considering several cases as listed below:
\begin{enumerate}
    \item[(i)] It is easy to see that
\begin{equation*}\label{LDfirstrow}
    (LD)[1] = D[1]-D[2] = (-1,0,\dotsc,0).
\end{equation*}

\item[(ii)] For $2\leq i \leq n-1$
\begin{equation*}\label{LD2ton-1row}
    \begin{aligned}
    (LD)[i] &= 2D[i]-D[i-1]-D[i+1] \\
    &=2(\underbrace{i,\dotsc,i}_{i},i+1,i+2,\dotsc,n,\dotsc,n)\\
    &~~~-(\underbrace{i-1,\dotsc,i-1}_{i-1},i,i+1,i+2,\dotsc,n,\dotsc,n)\\
    &~~~-(\underbrace{i+1,\dotsc,i+1}_{i+1},i+2,\dotsc,n,\dotsc,n)\\
    &=({0,\dotsc,0},\underbrace{-1}_{i},0,\dotsc,0).
    \end{aligned}
\end{equation*}

\item[(iii)] Next, we note that
\begin{equation*}\label{LDnrow}
    \begin{aligned}
    (LD)[n] &= 2D[n]-D[n-1]-D[n+1] \\
    &=2n\1'-(\underbrace{n-1,\dotsc,n-1}_{n-1},n,\dotsc,n)-(\underbrace{n,\dotsc,n}_{n},n-1,\dotsc,n-1)\\
    &=({1,\dotsc,1},\underbrace{0}_{n},1,\dotsc,1),\\
    \end{aligned}
\end{equation*}
and 
\begin{equation*}\label{LDn+1row}
    \begin{aligned}
    (LD)[n+1] &= 2D[n+1]-D[n]-D[n+2] \\
    &=2(\underbrace{n,\dotsc,n}_{n},n-1,\dotsc,n-1)-n\1'\\
    &~~~-(\underbrace{n,\dotsc,n}_{n},n-1,{n-2,\dotsc,n-2})\\
    &=({0,\dotsc,0},\underbrace{-1}_{n+1},0,\dotsc,0).
    \end{aligned}
\end{equation*}
\item[(iv)]
Let $n+2 \leq k \leq 2n-2$. Then
\begin{equation*}\label{LDn+2to2n-2row}
    \begin{aligned}
    (LD)[k] &= 2D[k]-D[k-1]-D[k+1] \\
    &=2(\underbrace{n,\dotsc,n}_{n},n-1,n-2,\dotsc,\underbrace{2n-k,\dotsc,2n-k}_{2n-k}) \\
    &~~~~-(\underbrace{n,\dotsc,n}_{n},n-1,n-2,\dotsc,\underbrace{2n-k+1,\dotsc,2n-k+1}_{2n-k+1})\\
    &~~~~-(\underbrace{n,\dotsc,n}_{n},n-1,n-2,\dotsc,\underbrace{2n-k-1,\dotsc,2n-k-1}_{2n-k-1})\\
    &=(0,\dotsc,0,\underbrace{-1}_{k},0,0,\dotsc,0).
    \end{aligned}
\end{equation*}

\item[(v)] Finally
\begin{equation*}\label{LD2n-1row}
    \begin{aligned}
    (LD)[2n-1] &= D[2n-1]-D[2n-2]\\
    &=(\underbrace{n,\dotsc,n}_{n},n-1,n-2,\dotsc,1)-(\underbrace{n,\dotsc,n}_{n},n-1,n-2,\dotsc,3,2,2)\\
    &=(0,\dotsc,0,-1).
    \end{aligned}
\end{equation*}
 \end{enumerate}
From (i), (ii), (iii), (iv) and (v), we note that except the $n^{\rm th}$ row, all other rows of $LD$ are equal to the corresonding rows of the matrix $-I$. Hence $LD+I$ has all other rows except the $n^{\rm th}$ row as zero. Also, the $n^{\rm th}$ row of $LD+I$ is the all one vector. Thus 
\[LD+I = u\1',\]
where $u = (0,\dotsc,0,\underbrace{1}_{n},0,\dotsc,0)'.$ This completes the proof.
\end{proof}
\begin{theorem}\label{path}
Let $D=D_2(P_n)[X,X]$ and $L$ be the Laplacian of a path graph on $2n-1$ vertices. Then
\[D^{-1} =-L+\frac{1}{n}uu',\]
where $u = (0,\dotsc,0,\underbrace{1}_{n},0,\dotsc,0) \in \rr^{2n-1}.$
\end{theorem}
\begin{proof}
Since the $n^{\rm th}$ row of $D$ is $n\1'$, it is easy to see that $u'D = n\1'$. Now 
\begin{equation*}
    \begin{aligned}
    (-L+\frac{1}{n}uu')D&=I-u\1'+\frac{1}{n}uu'D = I.
    \end{aligned}
\end{equation*}
Thus, $D$ is invertible and 
\[D^{-1} =-L+\frac{1}{n}uu'.\]
The proof is complete.
\end{proof}

\section{Caterpillar graphs} 
In this section, we first introduce caterpillar graphs and observe the structure of its $2-$Steiner distance matrix. Next, we compute the rank of its $2-$Steiner distance matrix by finding an invertible submatrix of maximum size. Let $\mathcal{C}$ be a caterpillar graph and let $P_n$ be a central path of maximum length in $\mathcal{C}$  (see Figure \ref{fig_caterpillar}). 
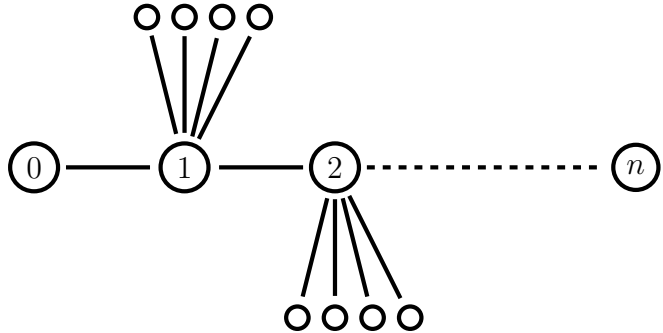
\begin{figure}[!h]
\centering
\begin{tikzpicture}[shorten >=1pt, auto, node distance=3cm, ultra thick,
   node_style/.style={circle,draw=black,fill=white !20!,font=\sffamily\Large\bfseries},
   edge_style/.style={draw=black, ultra thick}]
\node[vertex] (0) at  (0,0) {$0$};
\node[vertex] (1) at  (2,0) {$1$};
\node[vertex] (n+1) at  (2,2) {$ $};
\node[vertex] (n+2) at  (1.5,2) {$ $};
\node[vertex] (x) at  (2.5,2) {$ $};
\node[vertex] (y) at  (3,2) {$ $};
\node[vertex] (n+3) at  (4,-2) {$ $};
\node[vertex] (n+4) at  (3.5,-2) {$ $};
\node[vertex] (n+5) at  (4.5,-2) {$ $};
\node[vertex] (n+i2) at  (5,-2) {$ $};
\node[vertex] (2) at  (4,0) {$2$};  
\node[vertex] (n) at  (8,0) {$n$};  
\draw  (0) to (1);
\draw  (n+1) to (1);
\draw  (n+2) to (1);
\draw  (2) to (1);
\draw  (1) to (y);
\draw  (2) to (n+3);
\draw  (2) to (n+4);
\draw  (2) to (n+i2);
\draw  (1) to (x);
\draw  (2) to (n+5);
\begin{scope}[dashed]
\draw (2) to (n);
\end{scope}
\end{tikzpicture}
\caption{$\mathcal{C}$} \label{fig_caterpillar}
\end{figure}
Without loss of generality, we label the vertices of $\mathcal{C}$ in the following way:
\begin{enumerate}
    \item The vertices of $P_n$ are labeled from $0$ to $n$.
    \item The vertices adjacent to $1$ other than that of $P_n$ are labeled from $n+1,\dotsc, n+i_1$.
    \item The vertices adjacent to $2$ other than that of $P_n$ are labeled from $n+i_1+1,\dotsc, n+i_1+i_2$.
    \item Continuing like this, the vertices adjacent to $n-1$ other than that of $P_n$ are labeled from $n+\sum_{j=1}^{n-2}i_{j}+1,\dotsc, n+\sum_{j=1}^{n-1}i_{j}$.
\end{enumerate}
In the rest of the paper, we follow the above labelling for caterpillar graphs. We now illustrate the $2-$Steiner distance matrix of a caterpillar graph with an example.
\begin{example}\rm
Consider the caterpillar graph given in Figure \ref{fig_cater}. 
Suppose $D_2(\mathcal{C})$ denotes its $2-$Steiner distance matrix and let \[X= \{\{0,1\},\{0,2\},\{0,3\},\{0,4\},\{1,4\},\{2,4\},\{3,4\}, \{0,5\},\{0,6\}, \{0,7\}\}.\]
Then
\[D_2(\mathcal{C})[X,X] = \left[\begin{array}{cccccccccccc}
     1 & 2 & 3 & 4 & 4 & 4 & 4 & 2 & 3 & 3 \\
2 & 2 & 3 & 4 & 4 & 4 & 4 & 3 & 3 & 3 \\
3 & 3 & 3 & 4 & 4 & 4 & 4 & 4 & 4 & 4 \\
4 & 4 & 4 & 4 & 4 & 4 & 4 & 5 & 5 & 5 \\
4 & 4 & 4 & 4 & 3 & 3 & 3 & 5 & 5 & 5 \\
4 & 4 & 4 & 4 & 3 & 2 & 2 & 5 & 5 & 5 \\
4 & 4 & 4 & 4 & 3 & 2 & 1 & 5 & 5 & 5 \\
2 & 3 & 4 & 5 & 5 & 5 & 5 & 2 & 4 & 4 \\
3 & 3 & 4 & 5 & 5 & 5 & 5 & 4 & 3 & 4 \\
3 & 3 & 4 & 5 & 5 & 5 & 5 & 4 & 4 & 3
\end{array}\right].\]
We observe that the submatrix of $D_2(\mathcal{C})[X,X]$ corresponding to the first $7$ rows and columns is equal to the submatrix corresponding to the same rows and columns of the matrix $D_2(T)$ given in Example \ref{extree} .
\end{example}
\subsection{Block form for $D_2(\mathcal{C})[X,X]$}
We recall that $X = \{\{0,j\}|~j \neq 0\} \cup \{\{j,n\}|~deg(j)\neq 1\}.$ If we write elements of $X$ in the following order
\[X = \{\{0,1\}, \cdots, \{0,n\},\{1,n\},\cdots,\{n-1,n\}, \{0,n+1\},\{0,n+2\}, \dotsc, \{0,n+\sum_{j=1}^{n-1}i_{j}\}\},\]
then
\begin{enumerate}
    \item[(a)] If $D$ is the matrix given in Theorem $\ref{path}$, then $D$ is a principal submatrix of $D_2(\mathcal{C})[X,X]$. In fact, $D$ is the principal submatrix of $D_2(\mathcal{C})[X,X]$ corresponding to the first $2n-1$ rows and columns.
    \item[(b)] Let $X_1 = \{0,\a_1\}$ and $X_2 = \{0,\a_2\}$ be set of vertices of $\mathcal{C}$, where $1 \leq \a_1 \leq n$ and $\a_2 > n$ is adjacent to the vertex $\a$ of $P_n$. We note that
    \[d_{\mathcal{C}}(X_1 \cup X_2)= {\rm max}\{\a_1,\a\}+1.\]
    Suppose $X_3 = \{\a_3,n\}$, where $1 \leq \a_3 \leq n-1$. Then it is easy to see that
\[d_{\mathcal{C}}(X_3 \cup X_2)= n+1.\]
Now, let $X_4= \{0,\a_4\}$, where $\a_4 > n$ is adjacent to the vertex $\a'$ of $P_n$. Then
\[d_{\mathcal{C}}(X_2 \cup X_4) = \begin{cases}
\a+1&{\rm if}~\a_2=\a_4\\
\a+2&{\rm if}~\a_2\neq \a_4~{\rm but}~\a=\a'\\
{\rm max}\{\a,\a'\}+2&{\rm otherwise.}
\end{cases}\]

\item[(c)] Using observations (a) and (b), the submatrix $D_2(\mathcal{C})[X,X]$ of $D_2(\mathcal{C})$ can be written in the block form
    \begin{equation}\label{blockformC}
        D_2(\mathcal{C})[X,X] = \left[\begin{array}{cc}
    D & M \\
    M' & N
\end{array}\right],
    \end{equation}
where the matrix $M$ of order $(2n-1) \times \sum_{j=1}^{n-1} i_j$ and the matrix $N$ of order $\sum_{j=1}^{n-1} i_j \times \sum_{j=1}^{n-1} i_j$ are defined as follows:
For $1 \leq m \leq n-1$
\begin{equation}\label{Mrows}
    M[m]:=(m+1,\dotsc,m+1,\underbrace{m+2,\dotsc,m+2}_{i_{m+1}},\dotsc,\underbrace{n,\dotsc,n}_{i_{n-1}}),
\end{equation}
and for $n \leq m \leq 2n-1$
\begin{equation}\label{Mrow1}
    M[m] := (n+1) \1'.
\end{equation}
Before defining $N$, we note that any column of $N$ consist of the distance between the set $\{0,\alpha\}$, $\a>n$ and the sets
\[\{0,n+1\},\{0,n+2\}, \dotsc, \{0,n+\sum_{j=1}^{n-1}i_{j}\}.\]
Now, let $1 \leq m \leq \sum_{j=1}^{n-1} i_j$. For convenience, we assume $i_0 = 0$. Then there exists $k$, $1 \leq k \leq n-1$ such that $\sum_{j=0}^{k-1} i_j < m \leq \sum_{j=1}^{k} i_j$. Thus, the matrix $N$ is defined as follows:
\begin{equation}\label{N}
    \begin{aligned}
    m^{\rm th}~\mbox{column of}~ N &:=  (k+2,\dotsc,k+2,\underbrace{k+1}_{m^{\rm th}},k+2,\dotsc,k+2,\underbrace{k+3,\dotsc,k+3}_{i_{k+1}},\\
    &~~~~\dotsc ,\underbrace{n+1,\dotsc,n+1}_{i_{n-1}})'.
    \end{aligned}
\end{equation} 
\end{enumerate}

\subsection{Invertibility of $D_2(\mathcal{C})[X,X]$}
In this subsection, we show that $D_2(\mathcal{C})[X,X]$ is invertible. This implies ${\rank}(D_2(\mathcal{C})) \geq |X|$.
By Theorem \ref{path}, we know that $D$ is invertible and 
\[D^{-1} =-L+\frac{1}{n}uu',\]
where $L$ is the Laplacian matrix of a path graph on $2n-1$ vertices and $u = (0,\dotsc,0,\underbrace{1}_{n},0,\dotsc,0) \in \rr^{2n-1}$.
Suppose $P$ denotes the schur complement of $D$ in $D_2(\mathcal{C})[X,X]$. We know that
\begin{equation*}
\begin{aligned}
  P &= N-M'D^{-1}M  \\
  &= N-M'(-L+\frac{1}{n}uu')M \\
  &=N+M'LM-\frac{1}{n}M'uu'M.
\end{aligned}
\end{equation*}
Since $n^{\rm th}$ row of $M$ is $(n+1)\1'$, we conclude
that
\[u'M = (n+1)\1'.\]
Thus
\begin{equation}\label{P}
\begin{aligned}
  P =N+M'LM-\frac{(n+1)^2}{n}\1\1'.
\end{aligned}
\end{equation}
We claim that $D_2(\mathcal{C})[X,X]$ is invertible. Since $D$ is invertible, it is enough to show that $P$ is invertible. In the subsequent lemmas, we compute an exact expression for $P$.

\begin{lemma}\label{LM}
Suppose $M$ is the matrix given in the block form (\ref{blockformC}) of $D_2(\mathcal{C})[X,X]$ and $L$ is the Laplacian matrix of a path graph on $2n-1$ vertices. Then
\begin{equation*}
    (LM)[m] = \begin{cases}
    (0,\dotsc,0,\underbrace{-1,\dotsc,-1}_{i_{m}},0,\dotsc,0)&~\mbox{if}~1 \leq m \leq n-1\\
    (1, \dotsc,1)&~\mbox{if}~m=n\\
    (0,\dotsc,0)&\mbox{otherwise.}
    \end{cases}
\end{equation*}
\end{lemma}
\begin{proof}
We recall that $L$ is a tridiagonal matrix with its main diagonal determined by the vector $(1,2,\dotsc,2,1)$ and its first diagonal above and below the main diagonal have all the entries equal to $-1$. Using this observation, it is easy to deduce that
\[(LM)[1] = M[1]-M[2],\]
\[(LM)[m] = 2M[m]-M[m-1]-M[m+1],~\mbox{for}~2\leq m \leq 2n-2\]
and
\[(LM)[2n-1] = M[2n-1]-[2n-2].\]
Now, we compute $LM$ by repeatedly using (\ref{Mrows}) and (\ref{Mrow1}) in the following cases.
\begin{enumerate}
    \item[(i)] It is easy to see that
    \[(LM)[1] = M[1]-M[2] = (\underbrace{-1,\dotsc,-1}_{i_1},0,\dotsc,0).\]
\item[(ii)] Suppose $2 \leq m \leq n-2$. Then
\begin{equation*}
    \begin{aligned}
(LM)[m] &= 2M[m]-M[m-1]-M[m+1] \\
&= 2(m+1,\dotsc,m+1,\underbrace{m+2,\dotsc,m+2}_{i_{m+1}},\dotsc,\underbrace{n,\dotsc,n}_{i_{n-1}})    \\
&~~~~-({m,\dotsc,,m},\underbrace{m+1,\dotsc,m+1}_{i_{m}},\dotsc,\underbrace{n,\dotsc,n}_{i_{n-1}})\\
&~~~~-({m+2,\dotsc,m+2},\underbrace{m+3,\dotsc,m+3}_{i_{m+2}},\dotsc,\underbrace{n,\dotsc,n}_{i_{n-1}})\\
&= (0,\dotsc,0,\underbrace{-1,\dotsc,-1}_{i_{m}},0,\dotsc,0).
    \end{aligned}
\end{equation*}

\item[(iii)] Next,
\begin{equation*}
    \begin{aligned}
(LM)[n-1] &= 2M[n-1]-M[n-2]-M[n] \\
&= 2n\1'-({n-1,\dotsc,n-1},\underbrace{n,\dotsc,n}_{i_{n-1}})-(n+1) \1'\\
&=(0,\dotsc,0,\underbrace{-1,\dotsc,-1}_{i_{n-1}}),
    \end{aligned}
\end{equation*}
and
\begin{equation*}
    \begin{aligned}
(LM)[n] &= 2M[n]-M[n-1]-M[n+1] \\
&= 2(n+1) \1' - n\1'-(n+1) \1'\\
&=\1'.
    \end{aligned}
\end{equation*}

\item[(iv)]Suppose $n+1 \leq m \leq 2n-2$. Since $    M[m] = (n+1) \1'$, we have
\[(LM)[m] = 2M[m]-M[m-1]-M[m+1] =(0,\dotsc,0).\]
\item[(v)] Finally, 
\[(LM)[2n-1] = M[2n-1]-[2n-2] = (0,\dotsc,0).\]
\end{enumerate}
From (i), (ii), (iii), (iv) and (v), we deduce
\begin{equation*}
    (LM)[m] = \begin{cases}
    (0,\dotsc,0,\underbrace{-1,\dotsc,-1}_{i_{m}},0,\dotsc,0)&~\mbox{if}~1 \leq m \leq n-1\\
    (1, \dotsc,1)&~\mbox{if}~m=n\\
    (0,\dotsc,0)&\mbox{otherwise}
    \end{cases}
\end{equation*}
The proof is complete.
\end{proof}

\begin{lemma}\label{MLM}
Suppose $M$ is the matrix given in the block form (\ref{blockformC}) of $D_2(\mathcal{C})[X,X]$ and $L$ is the Laplacian matrix of a path graph on $2n-1$ vertices. If $1 \leq m \leq \sum_{j=1}^{n-1} i_j$ such that $\sum_{j=0}^{k-1} i_j < m \leq \sum_{j=1}^{k} i_j$, for some $1 \leq k \leq n-1$, then
\begin{equation*}
    m^{\rm th}~\mbox{column of}~ M'LM = (n-k,\dotsc,n-k,\underbrace{n-k-1,\dotsc,n-k-1}_{i_{k+1}},\dotsc,\underbrace{1,\dotsc,1}_{i_{n-1}})'.
\end{equation*}
\end{lemma}
\begin{proof}
Let $1 \leq m \leq \sum_{j=1}^{n-1} i_j$. Then there exists $k$, $1 \leq k \leq n-1$ such that $\sum_{j=0}^{k-1} i_j < m \leq \sum_{j=1}^{k} i_j$. Using Lemma \ref{LM}, (\ref{Mrows}) and (\ref{Mrow1}), we have
\begin{equation*}\label{MLM2}
    \begin{aligned}
    m^{\rm th}~\mbox{column of}~ M'LM &=-k^{\rm th}~\mbox{ column of }~M'+n^{\rm th}~\mbox{ column of }~M' \\ &=-M[k]'+M[n]' \\
    &=-({k+1,\dotsc,k+1},\underbrace{k+2,\dotsc,k+2}_{i_{k+1}},\dotsc,\underbrace{n,\dotsc,n}_{i_{n-1}})'+(n+1) \1\\
    &=(n-k,\dotsc,n-k,\underbrace{n-k-1,\dotsc,n-k-1}_{i_{k+1}},\dotsc,\underbrace{1,\dotsc,1}_{i_{n-1}})'.
    \end{aligned}
\end{equation*}
This completes the proof.
\end{proof}

\begin{lemma}\label{NMLM}
Suppose $M$ and $N$ are the matrices given in the block form (\ref{blockformC}) of $D_2(\mathcal{C})[X,X]$ and $L$ is the Laplacian matrix of a path graph on $2n-1$ vertices. Then
\[N+M'LM = (n+2)J-I.\]
\end{lemma}
\begin{proof}
Let $1 \leq m \leq \sum_{j=1}^{n-1} i_j$. Then there exists $k$, $1 \leq k \leq n-1$ such that $\sum_{j=0}^{k-1} i_j < m \leq \sum_{j=1}^{k} i_j$. Using Lemma \ref{MLM} and (\ref{N}), we have
\begin{equation*}
    \begin{aligned}
    m^{\rm th}~\mbox{column of}~ (N+M'LM) &= (k+2,\dotsc,k+2,\underbrace{k+1}_{m^{\rm th}},k+2,\dotsc,k+2,\underbrace{k+3,\dotsc,k+3}_{i_{k+1}},\\
    &~~~\dotsc,\underbrace{n+1,\dotsc,n+1}_{i_{n-1}})'\\
    &~~~+ (n-k,\dotsc,n-k,\underbrace{n-k-1,\dotsc,n-k-1}_{i_{k+1}},\dotsc,\underbrace{1,\dotsc,1}_{i_{n-1}})'\\
    &= (n+2,\dotsc,n+2,\underbrace{n+1}_{m^{\rm th}},n+2,\dotsc,n+2)'\\
    &= (n+2)\1-e_m.
    \end{aligned}
\end{equation*}
Here, $e_m \in \rr^{\sum_{j=1}^{n-1} i_j}$ with $1$ at its $m^{\rm th}$ position and $0$ elsewhere. Thus
\[N+M'LM = (n+2)J-I,\]
and the proof is complete.
\end{proof}

In the next theorem, we will prove the main result of this subsection.
\begin{theorem}\label{caterpillardet}
The matrix $D_2(\mathcal{C})[X,X]$ is invertible.
\end{theorem}
\begin{proof}
Using Lemma \ref{NMLM} in (\ref{P}), we have
\begin{equation*}
    \begin{aligned}
    P &= (n+2)J-I-\frac{1}{n}(n+1)^2J\\
    &= -\frac{1}{n}J-I.
    \end{aligned}
\end{equation*}
Thus $P$ is invertible. Since 
\[\det(D_2(\mathcal{C})[X,X]) = \det(D)\det(P),\]
we conclude that $D_2(\mathcal{C})[X,X]$ is invertible. The proof is complete.
\end{proof}

\subsection{Rank of $D_2(\mathcal{C})$}
For a distinct pair of vertices $\{k,j\}$, $D_2(\mathcal{C})'[\{k,j\}]$ is the column of $D_2(\mathcal{C})$ which is indexed by $\{k,j\}$. We recall that 
\[X = \{\{0,1\}, \cdots, \{0,n\},\{1,n\},\cdots,\{n-1,n\}, \{0,n+1\},\{0,n+2\}, \dotsc, \{0,n+\sum_{j=1}^{n-1}i_{j}\}\}.\]
We claim that each column of $D_2(\mathcal{C})$ corresponding to sets in $X^c$ is a linear combination of columns of $D_2(\mathcal{C})$ corresponding to the sets in $X$. This along with Theorem \ref{caterpillardet} proves that rank of $D_2(\mathcal{C})$ is $|X|$. Before, we begin the proof, we observe the following. Suppose $\{k,j\} \in X^c$. Then, exactly one of the following holds.
\begin{enumerate}
    \item[(i)] $0< k, j < n$,
    \item[(ii)] $0< k \leq n$ and $j>n$, or
    \item[(iii)]  $k,j>n$.
\end{enumerate}
In the subsequent lemmas, we discuss the above cases separately and prove that $D_2(\mathcal{C})'[\{k,j\}]$ is a linear combination of columns of $D_2(\mathcal{C})$ corresponding to the sets in $X$.
\begin{lemma}\label{l1}
Let $\{k,j\} \in X^c$. If $0< k, j < n$, then
\[D_2(\mathcal{C})'[\{k,j\}] = D_2(\mathcal{C})'[\{0,j\}]-D_2(\mathcal{C})'[\{0,n\}]+D_2(\mathcal{C})'[(k,n)].\]
\end{lemma} 
\begin{proof}
Without loss of generality, we assume $k<j$. Let $\{p,q\}$ be a set of distinct vertices of $\mathcal{C}$. We discuss the cases $\{p,q\} \in X$ and $\{p,q\} \in X^c$, separately. In the following tables, the entries in each column describes the entry in the $\{p,q\}^{\rm th}$ row of $D_2(\mathcal{C})'[\{\a,\a'\}]$, for different $\a$ and $\a'$.
\begin{enumerate}
    \item For $\{p,q\} \in X$, we have the following.
    
\begin{adjustbox}{width=0.94\textwidth}
\begin{tabular}{ |c|c|c|c|c| } 
 \hline
 & & & & \\
 $\{p,q\}$ & $D_2(\mathcal{C})'[\{0,j\}]$ & $D_2(\mathcal{C})'[\{0,n\}]$ &  $D_2(\mathcal{C})'[(k,n)]$ & $D_2(\mathcal{C})'[\{k,j\}]$\\ 
 & & & & \\
 \hline
 & & & & \\
 $\{0,l\}$, $0<l\leq n$ & $max\{l,j\}$ & $n$ &  $n$ & $max\{l,j\}$\\ 
 (Fig \ref{fig1})& & & & \\
 & & & & \\
 $\{l,n\}$, $0<l< n$ & $n$ & $n$ &  $n - min\{l,k\}$ & $n - min\{l,k\}$\\ 
 (Fig \ref{fig1})& & & & \\
 & & & & \\
 $\{0,l\}$, $l>n$ & $max\{i,j\}+1$ & $n+1$ &  $n+1$ & $max\{i,j\}+1$\\
 (Fig \ref{fig3})& & & & \\
 & & & & \\
 \hline
\end{tabular}
\end{adjustbox}

\begin{figure}[!h]
\centering
\begin{tikzpicture}[shorten >=1pt, auto, node distance=3cm, ultra thick,
   node_style/.style={circle,draw=black,fill=white !20!,font=\sffamily\Large\bfseries},
   edge_style/.style={draw=black, ultra thick}]
\node[vertex] (1) at  (2,0) {$0$};
\node[vertex] (l) at  (4,0) {$l$};  
\node[vertex] (k) at  (7,0) {$k$};  
\node[vertex] (j) at  (9,0) {$j$};  
\node[vertex] (n) at  (12,0) {$n$};  
\begin{scope}[dashed]
\draw (l) to (k);
\draw (j) to (k);
\draw (n) to (j);
\draw  (l) to (1);
\end{scope}
\end{tikzpicture}
\caption{} \label{fig1}
\end{figure}

\begin{figure}[!h]
\centering
\begin{tikzpicture}[shorten >=1pt, auto, node distance=3cm, ultra thick,
   node_style/.style={circle,draw=black,fill=white !20!,font=\sffamily\Large\bfseries},
   edge_style/.style={draw=black, ultra thick}]
\node[vertex] (1) at  (2,0) {$0$};
\node[vertex] (i) at  (4,0) {$i$};  
\node[vertex] (l) at  (4,2) {$l$};  
\node[vertex] (k) at  (7,0) {$k$};  
\node[vertex] (j) at  (9,0) {$j$};  
\node[vertex] (n) at  (12,0) {$n$};  
\draw  (i) to (l);
\begin{scope}[dashed]
\draw (i) to (k);
\draw (j) to (k);
\draw (n) to (j);
\draw  (i) to (1);
\end{scope}
\end{tikzpicture}
\caption{} \label{fig3}
\end{figure}

\item Suppose $\{p,q\} \in X^c$. Without loss of generality, we assume $p<q$. Now, the following holds.

\begin{adjustbox}{width=0.94\textwidth}
\begin{tabular}{ |c|c|c|c|c| } 
 \hline
 & & & & \\
 $\{p,q\}$ & $D_2(\mathcal{C})'[\{0,j\}]$ & $D_2(\mathcal{C})'[\{0,n\}]$ &  $D_2(\mathcal{C})'[(k,n)]$ & $D_2(\mathcal{C})'[\{k,j\}]$\\ 
 & & & & \\
 \hline
 & & & & \\
 $0<p,q< n$ & $max\{q,j\}$ & $n$ &  $n-min\{p,k\}$ & $max\{q,j\}$\\ 
(Fig \ref{fig4}) & & & &$-min\{p,k\}$ \\
 & & & & \\
 $0<p\leq n$, $q>n$ & $max\{p,i,j\}+1$ & $n+1$ &  $n-min\{p,i,k\}+1$ & $max\{p,i,j\}$\\ 
(Fig \ref{fig5})  & & & &$-min\{p,i,k\}+1$ \\
 & & & & \\
 $p,q> n$ & $max\{i,j\}+2$ & $n+2$ &  $n-min\{i',k\}+2$ & $max\{i,j\}$\\ 
(Fig \ref{fig9}) & & & &$-min\{i',k\}+2$ \\
 & & & & \\
 \hline
\end{tabular}
\end{adjustbox}
\begin{figure}[!h]
\centering
\begin{tikzpicture}[shorten >=1pt, auto, node distance=3cm, ultra thick,
   node_style/.style={circle,draw=black,fill=white !20!,font=\sffamily\Large\bfseries},
   edge_style/.style={draw=black, ultra thick}]
\node[vertex] (1) at  (2,0) {$0$};
\node[vertex] (p) at  (4,0) {$p$};  
\node[vertex] (q) at  (6,0) {$q$};  
\node[vertex] (k) at  (8,0) {$k$};  
\node[vertex] (j) at  (10,0) {$j$};  
\node[vertex] (n) at  (12,0) {$n$};  
\begin{scope}[dashed]
\draw (q) to (k);
\draw (j) to (k);
\draw (n) to (j);
\draw  (p) to (1);
\draw  (p) to (q);
\end{scope}
\end{tikzpicture}
\caption{} \label{fig4}
\end{figure}       
\begin{figure}[!h]
\centering
\begin{tikzpicture}[shorten >=1pt, auto, node distance=3cm, ultra thick,
   node_style/.style={circle,draw=black,fill=white !20!,font=\sffamily\Large\bfseries},
   edge_style/.style={draw=black, ultra thick}]
\node[vertex] (1) at  (2,0) {$0$};
\node[vertex] (p) at  (4,0) {$p$};  
\node[vertex] (i) at  (6,0) {$i$};  
\node[vertex] (q) at  (6,2) {$q$};  
\node[vertex] (k) at  (8,0) {$k$};  
\node[vertex] (j) at  (10,0) {$j$};  
\node[vertex] (n) at  (12,0) {$n$};  
\draw (i) to (q);
\begin{scope}[dashed]
\draw (i) to (k);
\draw (j) to (k);
\draw (n) to (j);
\draw  (p) to (1);
\draw  (p) to (i);
\end{scope}
\end{tikzpicture}
\caption{} \label{fig5}
\end{figure}   
  \begin{figure}[!h]
\centering
\begin{tikzpicture}[shorten >=1pt, auto, node distance=3cm, ultra thick,
   node_style/.style={circle,draw=black,fill=white !20!,font=\sffamily\Large\bfseries},
   edge_style/.style={draw=black, ultra thick}]
\node[vertex] (1) at  (2,0) {$0$};
\node[vertex] (i') at  (4,0) {$i'$};  
\node[vertex] (p) at  (4,2) {$p$};  
\node[vertex] (i) at  (6,0) {$i$};  
\node[vertex] (q) at  (6,2) {$q$};  
\node[vertex] (k) at  (8,0) {$k$};  
\node[vertex] (j) at  (10,0) {$j$};  
\node[vertex] (n) at  (12,0) {$n$};  
\draw (i) to (q);
\draw (i') to (p);
\begin{scope}[dashed]
\draw (i) to (k);
\draw (j) to (k);
\draw (n) to (j);
\draw  (i') to (1);
\draw  (i') to (i);
\end{scope}
\end{tikzpicture}
\caption{} \label{fig9}
\end{figure}   
\end{enumerate}
From Case 1 and 2, we conclude that
\[D_2(\mathcal{C})'[\{k,j\}] = D_2(\mathcal{C})'[\{0,j\}]-D_2(\mathcal{C})'[\{0,n\}]+D_2(\mathcal{C})'[(k,n)].\]
The proof is complete.
\end{proof} 

\begin{lemma}\label{l2}
Let $\{k,j\} \in X^c$. Suppose $0< k \leq n$, $j>n$ and $j$ is adjacent to vertex $i$ of $P_n$. If $k<i$, then
\[D_2(\mathcal{C})'[\{k,j\}] = D_2(\mathcal{C})'[\{0,j\}]-D_2(\mathcal{C})'[\{0,n\}]+D_2(\mathcal{C})'[(k,n)].\]
\end{lemma} 
\begin{proof}
Without loss of generality we assume $k<j$. We consider the following two cases.\begin{enumerate}
    \item Suppose $\{p,q\} \in X$. Then
    
\begin{adjustbox}{width=0.94\textwidth}
\begin{tabular}{ |c|c|c|c|c| } 
 \hline
 & & & & \\
 $\{p,q\}$ & $D_2(\mathcal{C})'[\{0,j\}]$ & $D_2(\mathcal{C})'[\{0,n\}]$ &  $D_2(\mathcal{C})'[(k,n)]$ & $D_2(\mathcal{C})'[\{k,j\}]$\\ 
 & & & & \\
 \hline
 & & & & \\
 $\{0,l\}$, $0<l\leq n$ & $max\{l,i\}+1$ & $n$ &  $n$ & $max\{l,i\}+1$\\ 
 (Fig \ref{fig10})& & & & \\
 & & & & \\
  $\{l,n\}$, $0<l< n$ & $n+1$ & $n$ &  $n - min\{l,k\}$ & $n - min\{l,k\}+1$\\ 
 (Fig \ref{fig10})& & & & \\
 & & & & \\
 $\{0,l\}$, $l>n$ & $max\{l',i\}+|\{l,j\}|$ & $n+1$ &  $n+1$ & $max\{l',i\}+|\{l,j\}|$\\
 (Fig \ref{fig6})& & & & \\
 & & & & \\
 \hline
\end{tabular}
\end{adjustbox}
\begin{figure}[!h]
\centering
\begin{tikzpicture}[shorten >=1pt, auto, node distance=3cm, ultra thick,
   node_style/.style={circle,draw=black,fill=white !20!,font=\sffamily\Large\bfseries},
   edge_style/.style={draw=black, ultra thick}]
\node[vertex] (1) at  (2,0) {$0$};
\node[vertex] (l) at  (4,0) {$l$};  
\node[vertex] (k) at  (7,0) {$k$};  
\node[vertex] (i) at  (9,0) {$i$};  
\node[vertex] (j) at  (9,2) {$j$};  
\node[vertex] (n) at  (12,0) {$n$};  
\draw (j) to (i);
\begin{scope}[dashed]
\draw (l) to (k);
\draw (i) to (k);
\draw (n) to (i);
\draw  (l) to (1);
\end{scope}
\end{tikzpicture}
\caption{} \label{fig10}
\end{figure}
\begin{figure}[!h]
\centering
\begin{tikzpicture}[shorten >=1pt, auto, node distance=3cm, ultra thick,
   node_style/.style={circle,draw=black,fill=white !20!,font=\sffamily\Large\bfseries},
   edge_style/.style={draw=black, ultra thick}]
\node[vertex] (1) at  (2,0) {$0$};
\node[vertex] (l') at  (4,0) {$l'$};  
\node[vertex] (l) at  (4,2) {$l$};  
\node[vertex] (k) at  (7,0) {$k$};  
\node[vertex] (i) at  (9,0) {$i$};  
\node[vertex] (j) at  (9,2) {$j$};  
\node[vertex] (n) at  (12,0) {$n$};  
\draw  (l') to (l);
\draw (j) to (i);
\begin{scope}[dashed]
\draw (l') to (k);
\draw  (l') to (1);
\draw (i) to (k);
\draw (n) to (i);
\end{scope}
\end{tikzpicture}
\caption{} \label{fig6}
\end{figure}

\item Suppose $\{p,q\} \in X^c$. Without loss of generality we assume $p<q$. Now,

\begin{adjustbox}{width=0.94\textwidth}
\begin{tabular}{ |c|c|c|c|c| } 
 \hline
 & & & & \\
 $\{p,q\}$ & $D_2(\mathcal{C})'[\{0,j\}]$ & $D_2(\mathcal{C})'[\{0,n\}]$ &  $D_2(\mathcal{C})'[(k,n)]$ & $D_2(\mathcal{C})'[\{k,j\}]$\\ 
 & & & & \\
 \hline
 & & & & \\
 $0<p,q< n$ & $max\{q,i\}$ & $n$ &  $n-min\{p,k\}$ & $max\{q,i\}$\\ 
(Fig \ref{fig7})&+1 & & &$-min\{p,k\}+1$ \\
 & & & & \\
 $0<p\leq n$, $q>n$ & $max\{p,q',i\}$ & $n+1$ &  $n-min\{p,q',k\}$ & $max\{p,q',i\}$\\ 
(Fig \ref{fig11})  &+$|\{q,j\}|$ & &$+1$ &$-min\{p,q',k\}+|\{q,j\}|$ \\
 & & & & \\
 $p,q> n$ & $max\{q',i\}$ & $n+2$ &  $n-min\{p',k\}$ & $max\{q',i\}$\\ 
 (Fig \ref{fig12})& $+|\{p,q,j\}|$& & $+2$&$-min\{p',k\}+|\{p,q,j\}|$ \\
 & & & & \\
 \hline
\end{tabular}
\end{adjustbox}
\begin{figure}[!h]
\centering
\begin{tikzpicture}[shorten >=1pt, auto, node distance=3cm, ultra thick,
   node_style/.style={circle,draw=black,fill=white !20!,font=\sffamily\Large\bfseries},
   edge_style/.style={draw=black, ultra thick}]
\node[vertex] (1) at  (2,0) {$0$};
\node[vertex] (p) at  (4,0) {$p$};  
\node[vertex] (q) at  (6,0) {$q$};  
\node[vertex] (k) at  (8,0) {$k$};  
\node[vertex] (i) at  (10,0) {$i$};  
\node[vertex] (j) at  (10,2) {$j$};  
\node[vertex] (n) at  (12,0) {$n$};  
\draw (j) to (i);
\begin{scope}[dashed]
\draw (q) to (k);
\draw (i) to (k);
\draw (n) to (i);
\draw  (p) to (1);
\draw  (p) to (q);
\end{scope}
\end{tikzpicture}
\caption{} \label{fig7}
\end{figure}      
\begin{figure}[!h]
\centering
\begin{tikzpicture}[shorten >=1pt, auto, node distance=3cm, ultra thick,
   node_style/.style={circle,draw=black,fill=white !20!,font=\sffamily\Large\bfseries},
   edge_style/.style={draw=black, ultra thick}]
\node[vertex] (1) at  (2,0) {$0$};
\node[vertex] (p) at  (4,0) {$p$};  
\node[vertex] (q') at  (6,0) {$q'$};  
\node[vertex] (q) at  (6,2) {$q$};  
\node[vertex] (k) at  (8,0) {$k$};  
\node[vertex] (i) at  (10,0) {$i$};  
\node[vertex] (j) at  (10,2) {$j$};  
\node[vertex] (n) at  (12,0) {$n$};  
\draw (j) to (i);
\draw (q') to (q);
\begin{scope}[dashed]
\draw (i) to (k);
\draw (q') to (k);
\draw (n) to (i);
\draw  (p) to (1);
\draw  (p) to (q');
\end{scope}
\end{tikzpicture}
\caption{} \label{fig11}
\end{figure}   
  \begin{figure}[!h]
\centering
\begin{tikzpicture}[shorten >=1pt, auto, node distance=3cm, ultra thick,
   node_style/.style={circle,draw=black,fill=white !20!,font=\sffamily\Large\bfseries},
   edge_style/.style={draw=black, ultra thick}]
\node[vertex] (1) at  (2,0) {$0$};
\node[vertex] (p') at  (4,0) {$p'$};  
\node[vertex] (p) at  (4,2) {$p$};  
\node[vertex] (q') at  (6,0) {$q'$};  
\node[vertex] (q) at  (6,2) {$q$};  
\node[vertex] (k) at  (8,0) {$k$};  
\node[vertex] (i) at  (10,0) {$i$};  
\node[vertex] (j) at  (10,2) {$j$};  
\node[vertex] (n) at  (12,0) {$n$};  
\draw (q') to (q);
\draw (p') to (p);
\draw (i) to (j);
\begin{scope}[dashed]
\draw (q') to (k);
\draw (i) to (k);
\draw (n) to (i);
\draw  (p') to (1);
\draw  (p') to (q');
\end{scope}
\end{tikzpicture}
\caption{} \label{fig12}
\end{figure}   
\end{enumerate}
From Case 1 and 2, we conclude that
\[D_2(\mathcal{C})'[\{k,j\}] = D_2(\mathcal{C})'[\{0,j\}]-D_2(\mathcal{C})'[\{0,n\}]+D_2(\mathcal{C})'[(k,n)].\]
This completes the proof.
\end{proof}

\begin{lemma}\label{l3}
Let $\{k,j\} \in X^c$. Suppose $0< k \leq n$, $j>n$ and $j$ is adjacent to vertex $i$ of $P_n$. If $k\geq i$, then
\begin{equation*}
    \begin{aligned}
    D_2(\mathcal{C})'[\{k,j\}] &= D_2(\mathcal{C})'[\{0,j\}]-D_2(\mathcal{C})'[\{0,n\}]+D_2(\mathcal{C})'[\{i,n\}]\\&~~~+D_2(\mathcal{C})'[\{0,k\}]-D_2(\mathcal{C})'[\{0,i\}].
    \end{aligned}
\end{equation*}
\end{lemma} 
\begin{proof}
Without loss of generality, we assume $k<j$. We consider the following two cases.
\begin{enumerate}
    \item Suppose $\{p,q\} \in X$. We compute the following.
    
\begin{adjustbox}{width=0.94\textwidth}
\begin{tabular}{ |c|c|c|c|c| } 
 \hline
 & & & & \\
 $\{p,q\}$ & $D_2(\mathcal{C})'[\{0,j\}]$ & $D_2(\mathcal{C})'[\{0,n\}]$ &  $D_2(\mathcal{C})'[\{i,n\}]$ & $D_2(\mathcal{C})'[\{0,k\}]$\\ 
 & & & & \\
 \hline
 & & & & \\
 $\{0,l\}$, $0<l\leq n$ & $max\{l,i\}+1$ & $n$ &  $n$ & $max\{l,k\}$\\ 
 (Fig \ref{fig13})& & & & \\
 & & & & \\
 $\{l,n\}$, $0<l< n$ & $n+1$ & $n$ &  $n-min\{l,i\}$ & $n $\\ 
 (Fig \ref{fig13})& & & & \\
 & & & & \\
 $\{0,l\}$, $l>n$ & $max\{l',i\}+|\{l,j\}|$ & $n+1$ &  $n+1$ & $max\{l',k\}+1$\\
 (Fig \ref{fig14})& & & & \\
 & & & & \\
 \hline
\end{tabular}
\end{adjustbox}
\begin{figure}[!h]
\centering
\begin{tikzpicture}[shorten >=1pt, auto, node distance=3cm, ultra thick,
   node_style/.style={circle,draw=black,fill=white !20!,font=\sffamily\Large\bfseries},
   edge_style/.style={draw=black, ultra thick}]
\node[vertex] (1) at  (2,0) {$0$};
\node[vertex] (l) at  (4,0) {$l$};  
\node[vertex] (k) at  (9,0) {$k$};  
\node[vertex] (i) at  (7,0) {$i$};  
\node[vertex] (j) at  (7,2) {$j$};  
\node[vertex] (n) at  (12,0) {$n$};  
\draw (j) to (i);
\begin{scope}[dashed]
\draw (l) to (i);
\draw (i) to (k);
\draw (n) to (k);
\draw  (l) to (1);
\end{scope}
\end{tikzpicture}
\caption{} \label{fig13}
\end{figure}
\begin{figure}[!h]
\centering
\begin{tikzpicture}[shorten >=1pt, auto, node distance=3cm, ultra thick,
   node_style/.style={circle,draw=black,fill=white !20!,font=\sffamily\Large\bfseries},
   edge_style/.style={draw=black, ultra thick}]
\node[vertex] (1) at  (2,0) {$0$};
\node[vertex] (l') at  (4,0) {$l'$};  
\node[vertex] (l) at  (4,2) {$l$};  
\node[vertex] (k) at  (9,0) {$k$};  
\node[vertex] (i) at  (7,0) {$i$};  
\node[vertex] (j) at  (7,2) {$j$};  
\node[vertex] (n) at  (12,0) {$n$};  
\draw  (l') to (l);
\draw (j) to (i);
\begin{scope}[dashed]
\draw (l') to (i);
\draw  (l') to (1);
\draw (i) to (k);
\draw (n) to (k);
\end{scope}
\end{tikzpicture}
\caption{} \label{fig14}
\end{figure}
\begin{small}
\begin{center}
\begin{tabular}{ |c|c|c| } 
 \hline
 & &  \\
 $\{p,q\}$ & $D_2(\mathcal{C})'[\{0,i\}]$ & $D_2(\mathcal{C})'[\{k,j\}]$\\ 
 & &  \\
 \hline
 & &  \\
 $\{0,l\}$, $0<l\leq n$ & $max\{l,i\}$ & $max\{l,k\}+1$\\ 
 (Fig \ref{fig13})& &  \\
 & &  \\
 $\{l,n\}$, $0<l< n$ & $n$ & $n - min\{l,i\}+1$\\ 
 (Fig \ref{fig13})& & \\
 & &  \\
 $\{0,l\}$, $l>n$ & $max\{l',i\}+1$ & $max\{l',k\}+|\{l,j\}|$\\
 (Fig \ref{fig14})& &  \\
 & &  \\
 \hline
\end{tabular}
\end{center}
\end{small} 
\item Suppose $\{p,q\} \in X^c$. Without loss of generality, we assume $p<q$. So, we have

\begin{figure}[!h]
\centering
\begin{tikzpicture}[shorten >=1pt, auto, node distance=3cm, ultra thick,
   node_style/.style={circle,draw=black,fill=white !20!,font=\sffamily\Large\bfseries},
   edge_style/.style={draw=black, ultra thick}]
\node[vertex] (1) at  (2,0) {$0$};
\node[vertex] (p) at  (4,0) {$p$};  
\node[vertex] (q) at  (6,0) {$q$};  
\node[vertex] (k) at  (10,0) {$k$};  
\node[vertex] (i) at  (8,0) {$i$};  
\node[vertex] (j) at  (8,2) {$j$};  
\node[vertex] (n) at  (12,0) {$n$};  
\draw (j) to (i);
\begin{scope}[dashed]
\draw (q) to (i);
\draw (i) to (k);
\draw (n) to (k);
\draw  (p) to (1);
\draw  (p) to (q);
\end{scope}
\end{tikzpicture}
\caption{} \label{fig15}
\end{figure}   
 \begin{adjustbox}{width=0.94\textwidth}
   \begin{tabular}{ |c|c|c|c|c| } 
 \hline
 & & & & \\
 $\{p,q\}$ & $D_2(\mathcal{C})'[\{0,j\}]$ & $D_2(\mathcal{C})'[\{0,n\}]$ &  $D_2(\mathcal{C})'[\{i,n\}]$ & $D_2(\mathcal{C})'[\{0,k\}]$\\ 
 & & & & \\
 \hline
 & & & & \\
 $0<p,q< n$ & $max\{q,i\}$ & $n$ &  $n-min\{p,i\}$ & $max\{q,k\}$\\ 
(Fig \ref{fig15})&+1 & & & \\
 & & & & \\
 $0<p\leq n$, $q>n$ & $max\{p,q',i\}$ & $n+1$ &  $n-min\{p,q',i\}$ & $max\{p,q',k\}+1$\\ 
  (Fig \ref{fig16})&+$|\{q,j\}|$ & &$+1$ & \\
 & & & & \\
 $p,q> n$& $max\{q',i\}$ & $n+2$ &  $n-min\{p',i\}$ & $max\{q',k\}+2$\\ 
 (Fig \ref{fig17})& $+|\{p,q,j\}|$& & $+2$& \\
 & & & & \\
 \hline
\end{tabular}
\end{adjustbox}
\begin{figure}[!h]
\centering
\begin{tikzpicture}[shorten >=1pt, auto, node distance=3cm, ultra thick,
   node_style/.style={circle,draw=black,fill=white !20!,font=\sffamily\Large\bfseries},
   edge_style/.style={draw=black, ultra thick}]
\node[vertex] (1) at  (2,0) {$0$};
\node[vertex] (p) at  (4,0) {$p$};  
\node[vertex] (q') at  (6,0) {$q'$};  
\node[vertex] (q) at  (6,2) {$q$};  
\node[vertex] (k) at  (10,0) {$k$};  
\node[vertex] (i) at  (8,0) {$i$};  
\node[vertex] (j) at  (8,2) {$j$};  
\node[vertex] (n) at  (12,0) {$n$};  
\draw (j) to (i);
\draw (q') to (q);
\begin{scope}[dashed]
\draw (i) to (k);
\draw (q') to (i);
\draw (n) to (k);
\draw  (p) to (1);
\draw  (p) to (q');
\end{scope}
\end{tikzpicture}
\caption{} \label{fig16}
\end{figure}   
\begin{figure}[!h]
\centering
\begin{tikzpicture}[shorten >=1pt, auto, node distance=3cm, ultra thick,
   node_style/.style={circle,draw=black,fill=white !20!,font=\sffamily\Large\bfseries},
   edge_style/.style={draw=black, ultra thick}]
\node[vertex] (1) at  (2,0) {$0$};
\node[vertex] (p') at  (4,0) {$p'$};  
\node[vertex] (p) at  (4,2) {$p$};  
\node[vertex] (q') at  (6,0) {$q'$};  
\node[vertex] (q) at  (6,2) {$q$};  
\node[vertex] (k) at  (10,0) {$k$};  
\node[vertex] (i) at  (8,0) {$i$};  
\node[vertex] (j) at  (8,2) {$j$};  
\node[vertex] (n) at  (12,0) {$n$};  
\draw (q') to (q);
\draw (p') to (p);
\draw (i) to (j);
\begin{scope}[dashed]
\draw (q') to (i);
\draw (i) to (k);
\draw (n) to (k);
\draw  (p') to (1);
\draw  (p') to (q');
\end{scope}
\end{tikzpicture}
\caption{} \label{fig17}
\end{figure}   
\begin{small}
\begin{center}
\begin{tabular}{ |c|c|c| } 
 \hline
 & &  \\
 $\{p,q\}$ & $D_2(\mathcal{C})'[\{0,i\}]$ & $D_2(\mathcal{C})'[\{k,j\}]$\\ 
 & & \\
 \hline
 & & \\
 $0<p,q< n$ & $max\{q,i\}$ & $max\{q,k\}-min\{p,i\}+1$\\ 
(Fig \ref{fig15}) &  &  \\
 & &  \\
 $0<p\leq n$, $q>n$  & $max\{p,q',i\}+1$ &  $max\{p,q',k\}-min\{p,q',i\}+|\{q,j\}|$\\ 
(Fig \ref{fig16}) & & \\
 & &  \\
 $p,q> n$& $max\{q',i\}+2$ &  $max\{q',k\}-min\{p',i\}+|\{p,q,j\}|$\\ 
(Fig \ref{fig17}) &  & \\
 & &  \\
 \hline
\end{tabular}
\end{center}
\end{small}
\end{enumerate}
    From Case 1 and 2, we conclude that
\begin{equation*}
    \begin{aligned}
    D_2(\mathcal{C})'[\{k,j\}] &= D_2(\mathcal{C})'[\{0,j\}]-D_2(\mathcal{C})'[\{0,n\}]+D_2(\mathcal{C})'[\{i,n\}]\\&~~~+D_2(\mathcal{C})'[\{0,k\}]-D_2(\mathcal{C})'[\{0,i\}].
    \end{aligned}
\end{equation*}
The proof is complete.
\end{proof}

\begin{lemma}\label{l4}
Let $\{k,j\} \in X^c$. Suppose $k,j>n$. If $j$ is adjacent to vertex $i$ of $P_n$ and $k$ is adjacent to vertex $k'$ of $P_n$, then
\begin{equation*}
    \begin{aligned}
    D_2(\mathcal{C})'[\{k,j\}] &= D_2(\mathcal{C})'[\{0,j\}]-D_2(\mathcal{C})'[\{0,n\}]+D_2(\mathcal{C})'[\{k',n\}]\\&~~~+D_2(\mathcal{C})'[\{0,k\}]-D_2(\mathcal{C})'[(0,k')].
    \end{aligned}
\end{equation*}
\end{lemma} 
\begin{proof}
Without loss of generality, we assume $k<j$. We consider the following two cases.
\begin{enumerate}
    \item If $\{p,q\} \in X$, then
    
\begin{adjustbox}{width=0.95\textwidth}
\begin{tabular}{ |c|c|c|c|c| } 
 \hline
 & & & & \\
 $\{p,q\}$ & $D_2(\mathcal{C})'[\{0,j\}]$ & $D_2(\mathcal{C})'[\{0,n\}]$ &  $D_2(\mathcal{C})'[\{k',n\}]$ & $D_2(\mathcal{C})'[\{0,k\}]$\\ 
 & & & & \\
 \hline
 & & & & \\
 $\{0,l\}$, $0<l\leq n$ & $max\{l,i\}+1$ & $n$ &  $n$ & $max\{l,k'\}+1$\\ 
 (Fig \ref{fig18})& & & & \\
 & & & & \\
 $\{l,n\}$, $0<l< n$ & $n+1$ & $n$ &  $n-min\{l,k'\}$ & $n+1 $\\
 (Fig \ref{fig18})& & & & \\
 
 & & & & \\
 $\{0,l\}$, $l>n$ & $max\{l',i\}+|\{l,j\}|$ & $n+1$ &  $n+1$ & $max\{l',k'\}+|\{l,k\}|$\\
 (Fig \ref{fig19})& & & & \\
 & & & & \\
 \hline
\end{tabular}
\end{adjustbox}
\begin{figure}[!h]
\centering
\begin{tikzpicture}[shorten >=1pt, auto, node distance=3cm, ultra thick,
   node_style/.style={circle,draw=black,fill=white !20!,font=\sffamily\Large\bfseries},
   edge_style/.style={draw=black, ultra thick}]
\node[vertex] (1) at  (2,0) {$0$};
\node[vertex] (l) at  (4,0) {$l$};  
\node[vertex] (k') at  (7,0) {$k'$};  
\node[vertex] (k) at  (7,2) {$k$};  
\node[vertex] (i) at  (9,0) {$i$};  
\node[vertex] (j) at  (9,2) {$j$};  
\node[vertex] (n) at  (12,0) {$n$};  
\draw (j) to (i);
\draw (k') to (k);
\begin{scope}[dashed]
\draw (l) to (k');
\draw (i) to (k');
\draw (n) to (i);
\draw  (l) to (1);
\end{scope}
\end{tikzpicture}
\caption{} \label{fig18}
\end{figure}
\begin{figure}[!h]
\centering
\begin{tikzpicture}[shorten >=1pt, auto, node distance=3cm, ultra thick,
   node_style/.style={circle,draw=black,fill=white !20!,font=\sffamily\Large\bfseries},
   edge_style/.style={draw=black, ultra thick}]
\node[vertex] (1) at  (2,0) {$0$};
\node[vertex] (l') at  (4,0) {$l'$};  
\node[vertex] (l) at  (4,2) {$l$};  
\node[vertex] (k') at  (7,0) {$k'$};  
\node[vertex] (k) at  (7,2) {$k$};  
\node[vertex] (i) at  (9,0) {$i$};  
\node[vertex] (j) at  (9,2) {$j$};  
\node[vertex] (n) at  (12,0) {$n$};  
\draw  (l') to (l);
\draw (j) to (i);
\draw (k') to (k);
\begin{scope}[dashed]
\draw (l') to (k');
\draw  (l') to (1);
\draw (i) to (k');
\draw (n) to (i);
\end{scope}
\end{tikzpicture}
\caption{} \label{fig19}
\end{figure}
\begin{small}
\begin{center}
\begin{tabular}{ |c|c|c| } 
 \hline
 & &  \\
 $\{p,q\}$ & $D_2(\mathcal{C})'[(0,k')]$ & $D_2(\mathcal{C})'[\{k,j\}]$\\ 
 & &  \\
 \hline
 & &  \\
 $\{0,l\}$, $0<l\leq n$ & $max\{l,k'\}$ & $max\{l,i\}+2$\\ 
 (Fig \ref{fig18})& &  \\
 & &  \\
 $\{l,n\}$, $0<l< n$ & $n$ & $n - min\{l,k'\}+2$\\ 
 (Fig \ref{fig18})& & \\
 & &  \\
 $\{0,l\}$, $l>n$ & $max\{l',k'\}+1$ & $max\{l',i\}+|\{l,j,k\}|$\\
 (Fig \ref{fig19})& &  \\
 & &  \\
 \hline
\end{tabular}
\end{center}
\end{small} 
Here, we observe the fact that $|\{l,j\}|+|\{l,k\}|-1=|\{l,j,k\}|$.
\item Suppose $\{p,q\} \in X^c$. Without loss of generality, we assume $p<q$. Now, we have

\begin{figure}[!h]
\centering
\begin{tikzpicture}[shorten >=1pt, auto, node distance=3cm, ultra thick,
   node_style/.style={circle,draw=black,fill=white !20!,font=\sffamily\Large\bfseries},
   edge_style/.style={draw=black, ultra thick}]
\node[vertex] (1) at  (2,0) {$0$};
\node[vertex] (p) at  (4,0) {$p$};  
\node[vertex] (q) at  (6,0) {$q$};  
\node[vertex] (k') at  (8,0) {$k'$};  
\node[vertex] (k) at  (8,2) {$k$};  
\node[vertex] (i) at  (10,0) {$i$};  
\node[vertex] (j) at  (10,2) {$j$};  
\node[vertex] (n) at  (12,0) {$n$};  
\draw (j) to (i);
\draw (k') to (k);
\begin{scope}[dashed]
\draw (q) to (k');
\draw (i) to (k');
\draw (n) to (i);
\draw  (p) to (1);
\draw  (p) to (q);
\end{scope}
\end{tikzpicture}
\caption{} \label{fig20}
\end{figure}      
\begin{figure}[!h]
\centering
\begin{tikzpicture}[shorten >=1pt, auto, node distance=3cm, ultra thick,
   node_style/.style={circle,draw=black,fill=white !20!,font=\sffamily\Large\bfseries},
   edge_style/.style={draw=black, ultra thick}]
\node[vertex] (1) at  (2,0) {$0$};
\node[vertex] (p) at  (4,0) {$p$};  
\node[vertex] (q') at  (6,0) {$q'$};  
\node[vertex] (q) at  (6,2) {$q$};  
\node[vertex] (k') at  (8,0) {$k'$};  
\node[vertex] (k) at  (8,2) {$k$};  
\node[vertex] (i) at  (10,0) {$i$};  
\node[vertex] (j) at  (10,2) {$j$};  
\node[vertex] (n) at  (12,0) {$n$};  
\draw (j) to (i);
\draw (q') to (q);
\draw (k') to (k);
\begin{scope}[dashed]
\draw (i) to (k');
\draw (q') to (k');
\draw (n) to (i);
\draw  (p) to (1);
\draw  (p) to (q');
\end{scope}
\end{tikzpicture}
\caption{} \label{fig21}
\end{figure}   
\begin{adjustbox}{width=0.94\textwidth}
\begin{tabular}{ |c|c|c|c|c| } 
 \hline
 & & & & \\
 $\{p,q\}$ & $D_2(\mathcal{C})'[\{0,j\}]$ & $D_2(\mathcal{C})'[\{0,n\}]$ &  $D_2(\mathcal{C})'[\{k',n\}]$ & $D_2(\mathcal{C})'[\{0,k\}]$\\ 
 & & & & \\
 \hline
 & & & & \\
 $0<p,q< n$ & $max\{q,i\}$ & $n$ &  $n-min\{p,k'\}$ & $max\{q,k'\}+1$\\ 
 (Fig \ref{fig20})&+1 & & & \\
 & & & & \\
 $0<p\leq n$, $q>n$ & $max\{p,q',i\}$ & $n+1$ &  $n-min\{p,q',k'\}$ & $max\{p,q',k'\}$\\ 
  (Fig \ref{fig21})&+$|\{q,j\}|$ & &$+1$ &$+|\{q,k\}|$ \\
 & & & & \\
 $p,q> n$ & $max\{q',i\}$ & $n+2$ &  $n-min\{p',k'\}$ & $max\{q',k'\}$\\ 
(Fig \ref{fig22}) & $+|\{p,q,j\}|$& & $+2$& $+|\{p,q,k\}|$\\
 & & & & \\
 \hline
\end{tabular}
\end{adjustbox}
\begin{figure}[!h]
\centering
\begin{tikzpicture}[shorten >=1pt, auto, node distance=3cm, ultra thick,
   node_style/.style={circle,draw=black,fill=white !20!,font=\sffamily\Large\bfseries},
   edge_style/.style={draw=black, ultra thick}]
\node[vertex] (1) at  (2,0) {$0$};
\node[vertex] (p') at  (4,0) {$p'$};  
\node[vertex] (p) at  (4,2) {$p$};  
\node[vertex] (q') at  (6,0) {$q'$};  
\node[vertex] (q) at  (6,2) {$q$};  
\node[vertex] (k') at  (8,0) {$k'$};  
\node[vertex] (k) at  (8,2) {$k$};  
\node[vertex] (i) at  (10,0) {$i$};  
\node[vertex] (j) at  (10,2) {$j$};  
\node[vertex] (n) at  (12,0) {$n$};  
\draw (q') to (q);
\draw (p') to (p);
\draw (i) to (j);
\draw (k') to (k);
\begin{scope}[dashed]
\draw (q') to (k');
\draw (i) to (k');
\draw (n) to (i);
\draw  (p') to (1);
\draw  (p') to (q');
\end{scope}
\end{tikzpicture}
\caption{} \label{fig22}
\end{figure}   

\begin{small}
\begin{center}
\begin{tabular}{ |c|c|c| } 
 \hline
 & &  \\
 $\{p,q\}$ & $D_2(\mathcal{C})'[(0,k')]$ & $D_2(\mathcal{C})'[\{k,j\}]$\\ 
 & & \\
 \hline
 & & \\
 $0<p,q< n$ & $max\{q,k'\}$ & $max\{q,i\}-min\{p,k'\}+2$\\ 
(Fig \ref{fig20}) &  &  \\
 & &  \\
 $0<p\leq n$, $q>n$& $max\{p,q',k'\}+1$ &  $max\{p,q',i\}-min\{p,q',k'\}+|\{q,j,k\}|$\\ 
  (Fig \ref{fig21})& & \\
 & &  \\
 $p,q> n$& $max\{q',k'\}+2$ &  $max\{q',i\}-min\{p',k'\}+|\{p,q,j,k\}|$\\ 
 (Fig \ref{fig22})&  & \\
 & &  \\
 \hline
\end{tabular}
\end{center}
\end{small}
Here, we note that $|\{q,j\}|+|\{q,k\}|-1=|\{q,j,k\}|$ and $|\{p,q,j\}|+|\{p,q,k\}|-2=|\{p,q,j,k\}|$.
\end{enumerate}
From Case 1 and 2, we conclude that
\begin{equation*}
    \begin{aligned}
    D_2(\mathcal{C})'[\{k,j\}] &= D_2(\mathcal{C})'[\{0,j\}]-D_2(\mathcal{C})'[\{0,n\}]+D_2(\mathcal{C})'[\{k',n\}]\\&~~~+D_2(\mathcal{C})'[\{0,k\}]-D_2(\mathcal{C})'[(0,k')].
    \end{aligned}
\end{equation*}
The proof is complete.
\end{proof}
Now, we state and prove the main result of this paper. 
\begin{theorem}
Let $\mathcal{C}$ be a caterpillar graph on $N$ vertices and let  $D_2(\mathcal{C})$ be its $2-$ Steiner distance matrix. Then
\[{\rm rank}(D_2(\mathcal{C})) = 2N-p-1,\]
where $p$ denotes the number of pendant vertices of $\mathcal{C}$.
\end{theorem}
\begin{proof}
Suppose $\mathcal{C}$ is labelled as described in Figure \ref{fig_caterpillar}. Using Lemma \ref{l1}-\ref{l4}, we conclude that
\[{\rm rank}(D_2(\mathcal{C})) = 2n-1+\sum_{j=1}^{n-1} i_j,\]
Since 
\begin{equation*}
    \begin{aligned}
    2N-p-1 &= 2(n+1+\sum_{j=1}^{n-1} i_j)-(\sum_{j=1}^{n-1} i_j+2)-1 \\
    &=2n-1+\sum_{j=1}^{n-1} i_j,
    \end{aligned}
\end{equation*}
the proof is complete.
\end{proof}
We conclude this paper with the following open problem.
\vskip .2cm
\noindent \textbf{Open Problem:} If $T$ is a general tree on $n$ vertices, find $\rank(D_2(T))$. 
\section*{Acknowledgement}
The second author acknowledges the support of the Indian National Science Academy under the INSA Senior Scientist scheme.

 \bibliography{mybibfile} 
 
Ali Azimi \\
Department of Mathematics, University of Neyshabur, Neyshabur, Iran\\
E-mail address: ali.azimi61@gmail.com\\

R.B. Bapat \\
Theoretical Statisticsand Mathematics Unit, Indian Statistical Institute, Delhi, India\\E-mail address: rbb@isid.ac.in\\

Shivani Goel \\
Department of Mathematics, IISc Bangalore, Bangalore, India.\\E-mail address: shivani.goel.maths@gmail.com

\end{document}